\documentclass{amsart}
\usepackage{amssymb,amsfonts}
\usepackage{latexsym}
\usepackage{amsmath,mathrsfs}
\usepackage[all,arc]{xy}
\usepackage{enumerate}
\usepackage[english]{babel}
\usepackage[bookmarks=true]{hyperref}
\usepackage{tikz}
\usepackage{ytableau}
\usepackage{blkarray}

\usepackage{mathdots} 
\newtheorem{theo}{Theorem}[section]

\newtheorem{prop}[theo]{Proposition}
\newtheorem{lemm}[theo]{Lemma}

\theoremstyle{definition}
\newtheorem{defi}[theo]{Definition}

\theoremstyle{remark}
\newtheorem{rema}[theo]{Remark}

\newcommand{\bb}[1]{\mathbb{#1}}
\newcommand{\al}[1]{\mathcal{#1}}
\newcommand{\sr}[1]{\mathscr{#1}}
\newcommand{\ak}[1]{\mathfrak{#1}}
\newcommand{\gen}[1]{\left\langle #1\right\rangle}
\newcommand{\sm}{\smallsetminus}

\newcommand{\id}{{\rm id}}

\newcommand{\ra}{\rightarrow}
\newcommand{\lra}{\longrightarrow}
\newcommand{\x}[1]{\text{#1}}

\begin{document}
\title{Invariant vector bundles and Hitchin systems}
\author{Zakaria OUARAS \and Hacen ZELACI}
\address{CNRS-Laboratoire AGM-Universit\'{e} de Cergy-Pontoise, 2 Avenue Adolphe Chauvin, 95302, Cergy-Pontoise, France.}
\curraddr{}
\email{zakaria.ouaras@cyu.fr}
\address{Mathematical departement, El Oued University, N48, 39000, EL Oued, Algeria.}
\email{zelaci-hacen@univ-eloued.dz}
\date{\today}
\thanks{}

\begin{abstract}
Let $X\ra Y$ be a Galois cover with Galois group $\Gamma$, where $X$ and $Y$ are smooth complex projective curve of genus $\geqslant 2$. 
In this paper, we study the moduli spaces of semistable $\Gamma-$invariant vector bundles on $X$ and classify their connected components. We also study the Hitchin systems on these moduli spaces and determine their fibers in the smooth case.
\end{abstract}
\maketitle
\tableofcontents

\section{Introduction}



Parahoric torsors over smooth algebraic curves have been a subject of intense study over the past decade due to their rich geometric and representation-theoretic properties. These objects arise naturally in the context of moduli spaces of vector bundles and Higgs bundles, particularly in the study of Galois covers. In this paper, we focus on a specific aspect of this theory: \( \Gamma \)-invariant vector bundles over a smooth projective curve \( X \), where \( \Gamma \) is a finite group acting on \( X \) via a Galois cover \( \pi: X \to Y \).  \smallskip

This work is built on the previous work of the second author in \cite{Z},  where the Hitchin system for vector bundles invariant under the action of an involution on the curve  \( X \) (corresponding to the group \( \mathbb{Z}/2\mathbb{Z} \)) was introduced and studied.  We extend this study to the case of an arbitrary finite group \( \Gamma \). The goal of the paper is to study the moduli space of these bundles, classify their connected components, and study the  Hitchin system over these moduli spaces.\smallskip



The $\Gamma-$invariant vector bundles are strongly related to parabolic vector bundles. In fact, in \cite{MS},  Seshadri established a correspondence between the moduli space of \( \Gamma \)-invariant vector bundles of a fixed type over the curve \( X \) and the moduli space of parabolic vector bundles over the quotient curve \( Y = X / \Gamma \) of a specified parabolic structure associated to the fixed type.  More details can be found in Appendix \ref{Seshadri correspondence}.  The study of the Hitchin system on the moduli spaces of parabolic bundles was done \cite{XWX},  where the authors studied the cotangent space of the moduli space of parabolic bundles, known as the space of strongly parabolic Higgs bundles,  and showed that it forms a completely integrable system for any choice of parabolic weights. In particular, they showed that the parabolic spectral curves are generally smooth in the full-flag case and singular otherwise.\smallskip

Furthermore, the authors explore the Hitchin system for another class of parabolic Higgs bundles, known as weak parabolic Higgs bundles (the Higgs field preserves the filtration). The associated Hitchin system was studied by Logares and Martens in \cite{LM},  where they computed the number of connected components of the fibers of this system and investigated the Poisson structure of the space. \smallskip

The situation becomes more intricate in the case of \( \Gamma \)-invariant vector bundles. Here,  the moduli space exhibits a richer structure, with more cases in which the spectral curve remains smooth (see Theorem \ref{smooth}). In these smooth cases, we describe the fibers of the \( \Gamma \)-invariant Hitchin system as \( \Gamma \)-invariant line bundles over the spectral curve,  equipped with the induced \( \Gamma \)-action for a type that we describe.\smallskip

To state our main results, we need to introduce some notation. Let $\pi: X \to Y$ be a Galois cover of degree-$n$ between irreducible smooth algebraic curves with Galois group $\Gamma$. Denote the ramification divisor by $R = \sum_i m_i p_i \subset X$, and let $m = \deg R$. By the Hurwitz formula, we have 
\[
2(g_X - 1) = 2n(g_Y - 1) + m.
\]
A $\Gamma-$invariant vector bundle is a vector bundle $E$ such that $\gamma^*E\cong E$, for any $\gamma\in \Gamma$.
For each $\Gamma-$invariant vector bundle over $X$, we associate a topological invariant that is called a \emph{type} and denoted $\theta$ (see Section \ref{section2} for more details). Fixing such a type, we have a corresponding moduli space of $\Gamma-$invariant bundles, denoted $\al U_X^{\Gamma,\theta}(r,d)$. These moduli spaces have been constructed and studied in \cite{S}, \cite{MS} and in a more general context in \cite{BS}.\smallskip

In a first stage, we compute the dimension of these moduli spaces by studying the infinitesimal deformations of such bundles (Theorem (\ref{dim})). \smallskip

Then, we consider the Hitchin morphism on these moduli spaces. We show that it is equivariant with respect to the $\Gamma-$action and show the following:
\begin{theo}[Theorem (\ref{hitchin})]
    For any type $\theta$, there exists a subspace $\al W^\theta\subset \al W$ such that 
    \begin{itemize}
        \item [$(i)$] The Hitchin system induces a map \(\al H_{\theta}:T^*\al U_X^{\Gamma,\theta}(r,d)\ra \al W^\theta\).
        \item [$(ii)$] We have \(\dim \al W^{\theta}=\dim \al U_X^{\Gamma,\theta}(r,d)\).
    \end{itemize}
\end{theo}

We then study the smoothness of the spectral curves associated with spectral data $s\in \al W^\theta$. In particular, we determine the types $\theta$ such that, for general $s\in \al W^\theta$, the associated spectral curve is smooth. Moreover, we study the fibers of the Hichin morphism in the smooth case.  
\begin{theo}[Theorem (\ref{main1})]
    For any smooth type $\theta$, there exists a type $\tilde\theta$ of $\Gamma-$invariant line bundles on the spectral curve $X_s$, such that the fiber $\al H_\theta^{-1}(s)$ is identified with a non-empty open subset of ${\rm Pic}^{c,\tilde \theta}(X_s)$ and $$\dim{\rm Pic}^{c,\tilde \theta}(X_s)=\dim\al U_X^{\Gamma,\theta}(r,d).$$ 
    In particular, the map \(\al H_{\theta}:T^*\al U_X^{\Gamma,\theta}(r,d)\ra \al W^\theta\) is an algebraic completely integrable system.
\end{theo}
    We also show that the canonical pushforward rational map $${\rm Pic}^{c,\tilde \theta}(X_s)\lra \al U_X^{\Gamma,\theta}(r,d)$$ is dominant. In particular, we deduce the connectedness of this moduli space.\\

    The paper is organized as follows.  In the first section, we recall the definition of $\Gamma-$invariant vector bundles and describe the space of infinitesimal deformation of such bundles as a $\Gamma-$invariant cohomology group hence calculate the dimension of the moduli space by Lefschetz fixed point formula and representation of finite groups. In the second section, we describe the Hitchin base of $\Gamma-$invariant Higgs fields for a fixed type by calculating the valuation of the coefficients of the characteristic polynomial, we get a $\Gamma-$invariant Hitchin system. In the third section, we describe the types for which the spectral curve is smooth and show that in this case the system we get is a completely integrable algebraic system equivalently $\Gamma-$invariant BNR-correspondence. 
\bigskip

{\bf Acknowledgments.} 
The second author acknowledges support from the CNRS (Chaire Maurice Audin) for a research visit to the Laboratoire J.A. Dieudonné at Université Côte d'Azur in November 2022, when this project was initiated. We are also very grateful to Christian Pauly for his valuable discussions.
\smallskip
\section{$\Gamma-$invariant vector bundles and their moduli spaces}\label{section2}
Let $X\ra Y$ be a Galois cover with Galois group $\Gamma$ of order $n$. A $\Gamma-$\emph{linearization}, or simply linearization, is a family of isomorphisms $\psi_\gamma$ indexed by $\Gamma_p$ for $p\in R$, such that $$\psi_\gamma:\gamma^*E\stackrel{\sim}{\lra} E$$  and $\psi_{e}=\id_E$  and $\psi_{\gamma\eta}=\psi_\eta\circ \eta^*\psi_\gamma$ for any $\gamma,\eta\in \Gamma_p$. In particular, we have $$\psi_{\gamma^k}=\psi_\gamma\circ \gamma^*\psi_\gamma\circ\cdots\circ(\gamma^k)^*\psi_\gamma.$$

\begin{defi}
   A \emph{$\Gamma$-invariant vector bundle}  is a  vector bundle $E$ over $X$  equipped with a $\Gamma$-linearization.
\end{defi}

\begin{lemm}
    The canonical line bundle $K_X$ is a $\Gamma$-invariant line bundle.
\end{lemm}

\begin{proof}
By the Hurwitz theorem, we have 
\[
K_X \simeq \pi^* K_Y \otimes \mathcal{O}_X(R).
\]
Since $\mathcal{O}_X(R)$ is $\Gamma$-invariant, it follows that $K_X$ admits a $\Gamma-$linearization.
\end{proof}

The linearization of the canonical line bundle $K_X$ given in this lemma is called the \emph{canonical linearization}.

\begin{lemm}
Let $p, q \in R$ be such that $\pi(p) = \pi(q)$. Then the isotropy subgroups of $p$ and $q$ are conjugate. In particular, the fiber of a branch point $x = \pi(p)$ consists of exactly $n/e_p$ points, where $e_p$ is the ramification index of $p$.
\end{lemm}

Let $E$ be a $\Gamma$-invariant vector bundle, and assume that it is stable. For $p \in R$, denote by $\gamma_p \in \Gamma$ a generator of the isotropy subgroup $\Gamma_p$, which we assume to be of order $e_p$. It is well known that the isotropic subgroup of a Galois covering is cyclic; see \cite{Seshadri}. The linearization $\psi_{\gamma_p}$ at $p$ gives an automorphism of the fiber $E_p$ of order $e_p$. Let ${\rm U}_{e_p} = \{\xi_{p} = e^{2\pi i /e_p}, \dots, \xi_{p}^{e_p}=1\}$ denote the set of $e_p$-th roots of unity. 
We order these roots increasingly with respect to their arguments. 

Since the eigenvalues of $\psi_{\gamma_p}$ belong to ${\rm U}_{e_p}$, we associate to $\psi_{\gamma_p}$ the vector
\[
\theta_p = (k_{p,1}, \dots, k_{p,e_p}),
\]
where $k_{p,l}$ is the multiplicity of $\xi_{p}^{l}$ as an eigenvalue of $\psi_{\gamma_p}$, for $l = 1, \dots, e_p$.  
Note that some $k_{p,l}$ may vanish and that they satisfy 
\[
\sum_{l=1}^{e_p} k_{p,l} = r.
\]  
We denote by $f_p$ the number of nonzero $k_{p,l}$. Clearly, we have $f_p \leqslant e_p$. \smallskip

Let $(\psi_\gamma)$ be a linearization on $E$, and let $\xi_p = e^{2\pi i / e_p}$. For any \( j \in \{1, \dots, e_p\} \), we can construct another linearization by multiplying $\psi_\gamma$, for $\gamma\in \Gamma_p$, with \(\xi_p^j\). 
\smallskip

Now, the vectors associated with this new linearization are obtained from the vectors \( \theta_p = (k_{p,1}, \dots, k_{p,e_p}) \) by shifting all coordinates by $j$ positions to the right in a cyclic manner:
\[\theta^j_p = (k_{p,e_p+1-j}, k_{p,e_p+2-j}, \dots, k_{p,e_p}, k_{p,1}, \dots, k_{p,e_p-j}).\]
This defines an equivalence relation on such vectors:  
\[
(v_p)_{p\in R} \sim (w_p)_{p\in R} \quad \Longleftrightarrow \quad \exists j \text{ such that } (w_p^j)_{p\in R} = (v_p)_{p\in R}.
\]

\begin{defi}
We define the \emph{type} of the $\Gamma$-invariant vector bundle $E$, denoted $\theta$, as the equivalence class, modulo the above equivalence relation, of  ${(\theta_p)_{ p \in R}}$. We call $\theta_p$ the local type at $p$.
\end{defi}

To each type $\theta$, we associate a family of  Young diagrams, at each ramification point $p$, it has  $f_p$ rows:  
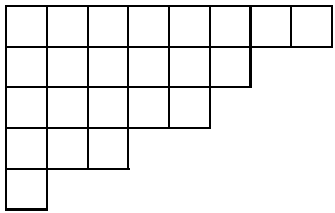
\[
\begin{ytableau}
  \, & \, & \, & \, & \, & \, & \, & \, \\
  \, & \, & \, & \, & \, & \, \\
  \, & \, & \, & \, & \, \\ 
  \, & \, & \, \\
  \, 
\end{ytableau}
\]
where the rows have lengths $k_{p,l}$ (ordered decreasingly).  \smallskip

Note that this correspondence is not injective, as the order of the integers $k_{p,l}$ is not preserved. This association is useful in determining the Hitchin base, as we will explain in Section~\ref{base}.

\begin{rema}
    Note that the canonical linearization on the canonical line bundle $K_X$ is equal to multiplication with $\xi_p$ in the fiber over any ramification point $p$. So its type is $\theta_p=(1,0,\cdots,0)$.
\end{rema}
\bigskip

The moduli space of semistable $(\Gamma,G)$-bundles was constructed and studied by Balaji-Seshadri in a more general context.\smallskip

In this section, we describe the space of infinitesimal deformations of $\Gamma$-invariant bundles and derive a formula for its dimension. We start by defining the semistability of these bundles. We denote by $\mathcal{U}_X^{\Gamma,\theta}(r,d)$ the moduli space of $\Gamma$-invariant vector bundles. \smallskip

Let $E$ be a $\Gamma$-invariant bundle. A subbundle $F$ of $E$ is called $\Gamma$-invariant if for any $\gamma \in \Gamma$, we have $\psi_\gamma(\gamma^*F) \subset F$. A $\Gamma$-invariant bundle $E$ is called semistable if for any proper subbundle $F$ (not necessarily $\Gamma-$invariant subbundle), we have
\[
    \mu(F) \leqslant \mu(E).
\]
A $\Gamma$-invariant bundle $E$ is called stable if for any proper $\Gamma-$invariant subbundle $F$, we have
 \[
 \mu(F) < \mu(E)
 \]
We denote by $\mathcal{U}_X^{\Gamma,\theta}(r,d)$ the moduli space of $\Gamma$-invariant vector bundles of type $\theta$.\smallskip

We saw in the last section that the canonical line bundle $K_X$ has a canonical linearization. Fixing a type $\theta$, for any $\Gamma$-invariant bundle $E$, the group $\Gamma$ and $\psi$ induce an action on $H^{i}(X,E)$ (for $i=0,1$). In particular, this action exists on $H^1(X, \text{End}(E))$ and its dual $H^0(X, \text{End}(E) \otimes K_X)$, where $K_X$ is equipped with its canonical linearization.

\begin{lemm}
    Serre duality is equivariant with respect to the canonical linearization on $K_X$.
\end{lemm}

\begin{proof}
Let $E$ be a $\Gamma$-invariant vector bundle on the curve $X$. Then Serre duality gives  
\[
    \left( H^1(X,E)^{\Gamma} \right)^* \simeq \left( H^1(Y,\pi^{\Gamma}_*(E) ) \right)^* \simeq  H^0(Y,\pi^{\Gamma}_*(E)^* \otimes K_Y ).
\]
Using the facts that 
\[
    \pi^{\Gamma}_*(K_X) \simeq K_Y, \quad \text{and} \quad \pi^{\Gamma}_*(E^* \otimes K_X) \simeq \pi^{\Gamma}_*(E)^* \otimes K_Y,
\]
we obtain  
\[
    H^0(Y,\pi^{\Gamma}_*(E)^* \otimes K_Y ) \simeq H^0(Y,\pi^{\Gamma}_*(E^* \otimes K_X) ).
\]
Hence,  
\[
    (H^0(X,E)^*)^\Gamma\simeq\left( H^1(X,E)^{\Gamma} \right)^* \simeq H^0(X,E^* \otimes K_X)^{\Gamma}.
\] 
Since the dual bundle $E^*$ is also $\Gamma$-invariant, its linearization is given by 
\[
    ^t\psi_\gamma^{-1}, \quad \forall \gamma \in \Gamma.
\]
\end{proof}

As a consequence of this lemma, we obtain the following isomorphism:
\[
    \left(H^1(X,\text{End}(E))^{\Gamma}\right)^* \simeq H^0(X,\text{End}(E) \otimes K_X)^{\Gamma},
\]
where $\text{End}(E) \otimes K_X$ is equipped with the natural $\Gamma$-action. An element in this space is called a \emph{$\Gamma$-invariant Higgs field} on $E$. \\

We can show this isomorphism using the cup product. We have the map:
\[
\begin{array}{ccccc}
\cup & : & H^1( X,{\rm End}(E))\otimes H^0\left(X,{\rm End}(E)K_X)\right)  & \longrightarrow &  H^1(X, K_X)\\
 & & f\otimes g& \longmapsto & f\cup g:= \mathrm{Tr}(f\circ g).
\end{array}
\] 
The cup product is \(\Gamma\)-equivariant, where the line bundle \( K_X \) is equipped with its natural linearization. Hence, we get a map on the \(\Gamma\)-invariant subspaces:
\[
\cup  :  H^1( X,{\rm End}(E))^{\Gamma}\otimes H^0\left(X,{\rm End}(E)K_X)\right)^{\Gamma}   \longrightarrow  H^1(X, K_X)^{\Gamma}.
\]
We calculate the right-hand side:
\[
 H^1(X, K_X)^{\Gamma} \simeq H^1(X, K_X(-R))^{\Gamma}=H^1(Y,K_Y) \simeq \mathbb{C}.
\]
Since the cup product is non-degenerate, this concludes the result.\\

Now we describe the infinitesimal deformation of $\Gamma-$invariant vector bundle.
\begin{theo} 
    The infinitesimal deformations of a $\Gamma$-invariant vector bundle $(E,\psi)$ are parametrized by $H^1(X,{\rm End}(E))^{\Gamma}$.  Moreover, the tangent space  
\[
    T_E \mathcal{U}_X^{\Gamma,\theta}(r,d) =  H^1(X,{\rm End}(E))^{\Gamma}.
\]
\end{theo}

\begin{proof} 
Recall that an infinitesimal deformation of a $\Gamma$-invariant vector bundle $(E,\psi)$ is an isomorphism class of $\Gamma$-invariant vector bundles $(E_{\epsilon},\psi_{\epsilon})$ over $X \times \text{Spec}(\mathbb{C}[\epsilon])$, where the $\Gamma$-action on $\text{Spec}(\mathbb{C}[\epsilon])$ is trivial, and the pullback with the natural embedding $X \hookrightarrow X \times \text{Spec}(\mathbb{C})$ is isomorphic as a $\Gamma$-bundle to $(E,\psi)$. 

Take an affine cover $\mathcal{U} = (U_{\lambda,\mu})$ of $X$ invariant under the $\Gamma$-action such that $E \vert_{U_{\lambda,\mu}} \simeq \mathcal{O}_X^{\oplus r}$. Then, over $U_{\lambda,\mu} = \text{Spec}(A_{\lambda,\mu})$, the vector bundle $E$ is given by an $A_{\lambda,\mu}$-module $M_{\lambda,\mu}$ equipped with a $\Gamma$-action. The deformation $E_{\epsilon}$ is given by transition maps of the form:
\[
    \tau_{\lambda,\mu}(m+\epsilon n) = m + \epsilon(\xi_{\lambda,\mu}(m) + n).
\]
This defines a $1$-cocycle in $H^1(X, \text{End}(E))$ that is also $\Gamma$-invariant, thus belonging to $H^1(X, \text{End}(E))^{\Gamma}$. Conversely, any such cocycle corresponds to a deformation of $(E,\psi)$.
\end{proof}
  
 \begin{theo}\label{dim} For any type $\theta$ we have:
     \begin{align*}
     \dim  \left( \al U_X^{\Gamma,\theta}(r,d) \right)&= \frac{r^2}{n}(g_X-1)+1+\frac{1}{2\vert \Gamma \vert}  \sum_{p \in R} \left( r^2- \sum^{e_p}_{i=1}  e_p k^2_{p,i} \right)\\
     &= r^2(g_Y-1)+1 +\frac{r^2}{2\vert \Gamma \vert} \deg(R+R_{red})- \sum_{p \in R}  \sum^{e_p}_{i=1} \frac{e_p  k^2_{p,i}}{2\vert \Gamma \vert} \\
     &=  r^2(g_Y-1)+1 + \frac{1}{2}  \sum_{q \in B} \left( r^2- \sum^{e_p}_{i=1}  k^2_{q,i} \right).
          \end{align*}
               where $B$ is the branched divisor of the Galois covering $\pi: X \longrightarrow Y$, and for all $q \in B$ we set:  $e_q:=e_p$ and $k_{q,i}:=k_{p,i}$ for some $p \in \pi^{-1}(q)$.
 \end{theo}
 To show this theorem we need to recall two results.
 \begin{theo}[\cite{serre1971representations}] Let $\Gamma$ be a finite group acting on a finite dimensional vector space $V$ equipped with a $\Gamma$-module structure given by group morphism $$ \rho :\Gamma \rightarrow \mathrm{End}(V).$$  
Let $V^{\Gamma}$ be the invariant subspace then we have  :  $$\dim V^{\Gamma}= \frac{1}{\vert \Gamma \vert} \sum_{\gamma \in \Gamma} Tr(\rho(\gamma)).$$
 \end{theo}
 Let $\gamma :X \longrightarrow X$ be an automorphism of the curve $X$ and $(E,\psi)$ a $\Gamma-$invariant vector bundle. Hence,  the map $\gamma$ induces endomorphisms $$\rho_i(\gamma): H^{i}(X,E) \longrightarrow H^{i}(X,E),$$ for $i=0,1$.  Define the Lefschetz index of $(\gamma,\psi)$ to be:
 $$L(\gamma,\psi):=Tr(\rho_0(\gamma))-Tr(\rho_1(\gamma)).$$ 
We recall Lefschetz fixed point formula.
 \begin{theo}[\cite{AB}, Theorem 4.12]
With the above notations, we have $$ L(\gamma,\psi)=\sum_{p \in {\rm Fix}(\gamma)} \frac{Tr(\psi_{\gamma,p})}{\det(1-d_p\gamma)}.$$
 \end{theo}
 \begin{proof}[Proof of Theorem \ref{dim}] Let $E$ be a stable vector bundle,  hence it is $\Gamma-$stable and $$H^0(X,{\rm End}(E))\cong \bb C.$$ Now a $\Gamma-$stable vector bundle is a smooth point in the moduli space $\al U_X^{\Gamma,\theta}(r,d)$,  so we have the equality: 
\begin{align*}
  \dim  \left( \al U_X^{\Gamma,\theta}(r,d) \right) &= \dim T_E\al U_X^{\Gamma,\theta}(r,d)  = \dim H^1(X,{\rm End}(E))^{\Gamma}.
\end{align*}
Let denote by  $V_i=H^i(X,{\rm End}(E))$ for $i=0,1$.   By   the previous theorems we get:
\begin{align*}
  \dim H^1(X,{\rm End}(E))^{\Gamma} &= \frac{1}{\vert \Gamma \vert} \sum_{\gamma \in \Gamma} Tr(\rho_1(\gamma))\\
  &= \frac{1}{\vert \Gamma \vert}  \dim H^1(X,{\rm End}(E))+ \frac{1}{\vert \Gamma \vert} \sum_{\gamma \in \Gamma-\{id\}} (Tr(\rho_0(\gamma))-L(\gamma,\psi)).
  \end{align*}
we use the fact that $\rho_1(\id)=\id:H^1(X,{\rm End}(E)) \rightarrow  H^1(X,{\rm End}(E))$ hence the trace is the dimension, and  as the vector bundle is stable we get for all $\gamma \in \Gamma$ that $\rho_0(\gamma)=1$, this is do to the fact that the representation $\rho_0$ is given by the the natural representation $\psi_{\gamma}\otimes \  ^t\psi_{\gamma}^{-1} $ on ${\rm End}(E)$ which is trivial globally. Hence by Riemann-Roch formula we get : 
\begin{align*}
  \dim H^1(X,{\rm End}(E))^{\Gamma} &= \frac{r^2(g_X-1)+1}{\vert \Gamma \vert}+ \frac{1}{\vert \Gamma \vert} \sum_{\gamma \in \Gamma-\{id\}} (1-L(\gamma,\psi)) \\
  &=  \frac{r^2(g_X-1)}{\vert \Gamma \vert}+1-\frac{1}{\vert \Gamma \vert} \sum_{\gamma \in \Gamma-\{id\}} \sum_{p \in {\rm Fix}(\gamma)} \frac{Tr(\psi_{\gamma,p})}{1-d_p\gamma}
    \end{align*}
    By Hurwitz formula we get the following equality: 
    \begin{align*}
    \dim H^1(X,{\rm End}(E))^{\Gamma} &= r^2(g_Y-1)+1+ \frac{r^2}{2 \vert \Gamma \vert} \deg(R)- \frac{1}{\vert \Gamma \vert} \sum_{p \in R} \sum_{\gamma \in \Gamma_p, \gamma\neq id} \frac{Tr(\psi_{\gamma,p})}{1-d_p\gamma} \\
    &= r^2(g_Y-1)+1+ \frac{r^2}{2 \vert \Gamma \vert} \deg(R)- \frac{1}{\vert \Gamma \vert} \sum_{p \in R} \sum_{\gamma \in \Gamma_p, \gamma \neq id} \frac{Tr(\psi_{\gamma,p})}{1-d_p\gamma}  \\
    \dim H^1(X,{\rm End}(E))^{\Gamma} &= r^2(g_Y-1)+1+ \frac{r^2}{2} \deg(B_{red}) - \frac{1}{\vert \Gamma \vert} \left( \frac{r^2}{2} \deg(R_{red})+ \sum_{p \in R} \sum_{\gamma \in \Gamma_p, \gamma \neq id}  L(\gamma,\psi) \right) \\
    &=r^2(g_Y-1)+1+ \frac{r^2}{2} \deg(B_{red}) - \frac{1}{\vert \Gamma \vert} \sum_{p \in R}  \left( \frac{r^2}{2}+  \sum_{\gamma \in \Gamma_p, \gamma\neq id}  L(\gamma,\psi) \right)
     \end{align*}
     Now to complete the calculation we apply for all $p \in R$ Lefschetz fix point Theorem for the representations of  the group $\Gamma_p$ given by
    \begin{align*}
     \widehat{\rho}_0&: \Gamma_p \hookrightarrow \Gamma \longrightarrow{\rm End} \left( H^0(X,{\rm End}(E)) \right) \simeq \bb C \\ 
     \widehat{\rho}_1&: \Gamma_p \hookrightarrow \Gamma \longrightarrow{\rm End} \left(  H^1(X,{\rm End}(E)) \right).
       \end{align*}
 For all $p \in  R$ we have:
   \begin{align*} 
  \frac{r^2}{2}+  \sum_{\gamma \in \Gamma_p, \gamma\neq id}  L(\gamma,\psi) &= \frac{r^2}{2}+\sum_{\gamma \in \Gamma_p, \gamma\neq id}  \left( Tr(\widehat{\rho}_0(\gamma))-Tr(\widehat{\rho}_1(\gamma)) \right) \\
  &=  \frac{r^2}{2}+ e_p-1+Tr(\widehat{\rho}_1(id))-\sum_{\gamma \in \Gamma_p}   Tr(\widehat{\rho}_1(\gamma)),
   \end{align*} 
   we used the fact that $\forall \gamma \in \Gamma_p$ we have  $\widehat{\rho}_0(\gamma)=1$ hence $ Tr(\widehat{\rho}_0(\gamma))=1$ and $Tr(\widehat{\rho}_1(id))=\dim  H^1(X,{\rm End}(E))$.  Now by Lefschetz Fix point Theorem for the group $\Gamma_p$,  we get 
     \begin{align*} 
      \frac{r^2}{2}+  \sum_{\gamma \in \Gamma_p, \gamma\neq id}  L(\gamma,\psi) &= \frac{r^2}{2}+ e_p-1+\dim  H^1(X,{\rm End}(E))-e_p \dim  H^1(X,{\rm End}(E))^{\Gamma_p} \\
      &=\frac{r^2}{2}+ e_p+r^2(g_X-1)-e_p \dim  H^1(X,{\rm End}(E))^{\Gamma_p}
     \end{align*} 
   Let calculate $\dim  H^1(X,{\rm End}(E))^{\Gamma_p}$.  Set $\pi_p:X \rightarrow X_p:=X/\Gamma_p$, we notice that the ramification is total in $p$,  then
     \begin{align*} 
     \dim  H^1(X,{\rm End}(E))^{\Gamma_p}&=   \dim  H^1(X_p, \mathcal{F})=r^2(g_{X_p}-1)+1+\deg(\mathcal{F})
         \end{align*} 
         where $ {\pi_p}^*\left( \mathcal{F} \right)$ is the Hecke modification of $End(E)$ for which the $\Gamma_p$-linearisation is trivial, one can show that $$\deg(\mathcal{F})=\frac{-1}{2} \left( r^2- \sum_{i=1}^{e_p} k^2_{p,i} \right)$$
  By applying Hurwitz formula for the covering $\pi_p:X \rightarrow X_p:=X/\Gamma_p$ we get:
      \begin{align*} 
        \dim  H^1(X,{\rm End}(E))^{\Gamma_p}&= \frac{1}{e_p} \left( r^2(g_X-1)-\frac{(e_p-1)r^2}{2} \right)+1-\deg(\mathcal{F}).
       \end{align*} 
       Hence 
        \begin{align*} 
      \frac{r^2}{2}+  \sum_{\gamma \in \Gamma_p, \gamma\neq id}  L(\gamma,\psi) &= \frac{r^2}{2}+ e_p+r^2(g_X-1)-e_p \dim  H^1(X,{\rm End}(E))^{\Gamma_p} \\
      &= \frac{r^2}{2}+\frac{(e_p-1)r^2}{2}+e_p\deg(\mathcal{F}) \\
      &=\frac{r^2}{2}+\frac{(e_p-1)r^2}{2}-\frac{e_pr^2}{2}+\frac{e_p}{2}  \sum_{i=1}^{e_p} k^2_{p,i} = \frac{e_p}{2}  \sum_{i=1}^{e_p} k^2_{p,i}
     \end{align*} 
     By substitution we get the result:
$$        \dim  H^1(X,{\rm End}(E))^{\Gamma}
=r^2(g_Y-1)+1 +\frac{1}{2} \sum_{q \in B} \left( r^2- \sum^{e_p}_{i=1}  k^2_{q,i} \right).
$$     
  \end{proof}
\section{$\Gamma-$invariant Hitchin system}\label{base}



The Hitchin map is a morphism defined on the cotangent space $T^*\al U_X(r,d)$ of the moduli space of vector bundles, and with values in the vector space $\bigoplus_{i=1}^r H^0(X,K_X^i)$. It associates to a pair $(E,\phi)$, where $\phi:E\ra E\otimes K_X$ is a Higgs field on $E$, the coefficients of its characteristic polynomial.  More precisely, let \(\al H\) be the Hitchin morphism :
\[
\al H: T^*\al U_X(r,d) \longrightarrow \al W:= \bigoplus_{i=1}^r H^0(X,K^i_X),
\]
it is given by
\[
\al H (E, \varphi):= \bigoplus_{i=1}^r \mathrm{Tr}( \wedge^{i}  \varphi )  \in \al W.
\]

In this section, we study this morphism in the $\Gamma-$invariant case, and in particular, we  determine the \(\Gamma\)-invariant Hitchin base.

\begin{lemm}
    The Hitchin map \(\al H\) is equivariant with respect to the canonical actions of \(\Gamma\).
\end{lemm}
\begin{proof} Let $\phi$ a $\Gamma-$ invariant Higgs bundles of type $\theta$, hence satisfies around $p \in R$ the equation: 
$$ \phi(t)= \xi \psi_{\gamma} \phi(\xi t) \psi^{-1}_{\gamma},$$
it follows that the endomorphisms $\phi(t)$ and $\xi\phi(\xi t)$ are similar, hence  they have the same characteristic polynomial, and we  get: $$\mathrm{Tr}(\wedge^{i} \varphi(t))=\mathrm{Tr}(\wedge^{i} \left( \xi \varphi(t) \right))=\xi^{i} \mathrm{Tr}(\wedge^{i}  \varphi(t)).$$  
 This concludes the proof as the linearization of the bundle $K_X^{i}$ is given by $\xi^{i}$.   \smallskip
\end{proof}
By this lemma, the restriction of \(\al H\) to the invariant part has values in the \(\Gamma\)-invariant space:
\[
\al H: T^*\al U_X^{\Gamma,\theta}(r,d) \longrightarrow \al W^\Gamma= \bigoplus_{i=1}^r H^0(X,K^i_X)^{\Gamma}.
\]
We show the following result.
\newpage
\begin{theo}\label{hitchin}
    For any type \(\theta\), there exists a subspace \(\al W^\theta \subset \al W^\Gamma\) such that:
    \begin{enumerate}
        \item [$(i)$] The Hitchin system induces a map: \[\al H_{\theta}:T^*\al U_X^{\Gamma,\theta}(r,d)\ra \al W^\theta.\]
        \item [$(ii)$] \; \(\dim \al W^{\theta}=\dim \al U_X^{\Gamma,\theta}(r,d)\).
    \end{enumerate}
\end{theo}  
\bigskip

To determine the subspace \( \al W^{\theta} \), we need to calculate the valuation of each part \( \mathrm{Tr}( \wedge^{i}  \varphi ) \) of the Hicthin map.  \smallskip

For \( p \in R \), the vector space \( E_p \) can be written as a direct sum of eigenspaces:
\[
E_p= \bigoplus^{\sigma}_{j=1} V_j.
\]
Let \( \mathcal{O}_p \) be the local ring around \( p \), and take a local coordinate \( t \) around \( p \), so that we have the identification 
\[
\mathcal{O}_p= \mathbb{C} \left[ \left[ t  \right] \right].
\] 
Let \( \phi \in H^0({\rm End}(E)K_X) \) be a \(\Gamma\)-invariant Higgs field, i.e., for all \( \gamma \in \Gamma \), the endomorphism satisfies the equation:
\[
\phi \circ \psi_{\gamma}=\psi_{\gamma} d\gamma\circ \gamma^*(\phi).
\]
In particular, around the point \( p \), if we take \( dt \) as a trivialization of \( K_{X,p} \), we get:
\begin{equation}\label{phi}
    \phi(t) = \xi ( \psi_{\gamma,p} \circ \phi(\xi t) \circ \psi^{-1}_{\gamma,p}).
\end{equation}

\begin{lemm}\label{formphi}
Let \( \phi \) be an endomorphism satisfying equation \eqref{phi}. Then the matrix representation of \( \phi \) is of the form
\[
A=(A_{i,j})_{0\leqslant i,j\leqslant f_p},
\]
where each \( A_{i,j} \) is a \( k_{p,l_i} \times k_{p,l_j} \) matrix, and whose each element \( a \) has valuation in \( t \) satisfying:
\[
\nu(a)\geqslant l_j - l_i -1 \pmod{e_p},
\]
where \( l_i \), for \( i=1,\dots,f_p \), are the indices of \( k_{p,l} \) that are nonzero.
\end{lemm}

\begin{proof} Let $\phi$ be a $\Gamma-$invariant Higgs field  with respect to a type $\theta$. We denote by $A$  the associated matrix around a ramification $p \in R$ with the variable $t$. Hence, by  Equation \eqref{phi} the $(i,j)$-block of $A$  satisfies:
    $$A_{i,j}(t)=\xi^{1+l_i-l_j}A_{i,j}(\xi t), $$
    therefore the elements of the block $A_{i,j}$ have the same valuation and one gets for any element $a$ of this block a lower bound of the valuation: $$\nu(a)\geqslant l_j - l_i -1 \ \text{mod $e_p$}.$$  
    This last inequality is a consequence of the following fact: let $\mathcal{O}_p\simeq\bb C[[t]]$ the local ring around $p$ and let $f(t)=\sum_{i \geqslant 0}a_i \ t^{i}$ be a local function around $p$, such that for some $s \in \{1,2, \dots e_p-1 \}$ we have:
    $$f(t)=\xi^sf(\xi t).$$
    Hence, the coefficients of $f$ for all $i$  satisfies the equation: $a_i=\xi^{s+i} a_i$, therefore
    $$s+i \not\equiv 0 \pmod{e_p} \implies a_i=0,$$
    finally there exists $g \in \bb C [[t]]$ such that 
    $$ f(t)=t^{\nu} g(t^{e_p}),$$ 
    where $\nu:=-s \pmod{e_p}$, in particular $\nu(f) \geqslant \nu$. This concludes the proof.  
    
\end{proof}

For all \( i\in\{1,\dots,r\} \), define the map \( \delta_p(i) \) as follows: For \( 1\leqslant j \leqslant  \max(k_{p,i}) \), set
\[
\delta_p(i):=j \quad \text{if} \quad \sum_{k=1}^{j-1} \mu_p(k)  < i\leq  \sum_{k=1}^{j} \mu_p(k),
\]
where 
\[
\mu_p(j):= \mathrm{Card} \{  k_{p,i} \geqslant j \}.
\]
This map can be defined using the Yang diagram, for example, for $r=23$, and  the type $\theta_p=(8,6,4,3,1)$, the values of $\delta_p$ are given in 
$$
\begin{ytableau}
   1& 2 & 3 & 4 & 5 & 6& 7 &8 \\
  1 & 2 & 3 & 4 & 5 &6 \\
1& 2 & 2 & 3& 4 \\ 
  1 & 2 & 3 \\
  1 
\end{ytableau} 
$$
where the order is from the above to the bottom, from the left to the right.

 \begin{theo} For all $p \in R$ and $i\in\{1,\cdots, r\} $  we have:
 $$ \nu(\mathrm{Tr}(\wedge^{i} \phi)) \geqslant e_p \delta_p(i)-i.$$
 \end{theo}
 \begin{proof}

We start with the case \( i = r \). We need to show that  
\[
\nu(\wedge^{r} \phi) \geqslant e_p \delta_p(r) - r.
\]  
It is sufficient to show that any summand of the determinant of \( \phi \) has valuation greater than or equal to \( e_p \delta_p(r) - r \). Each summand contains an element from each row and each column of the matrix of \( \phi \). Since the valuation of a summand is the sum of the valuations of its elements, Lemma (\ref{formphi}) implies that its valuation is at least \( me_p - r \), for some positive integer \( m \) depending on the chosen summand. Thus, it suffices to show that \( m \geqslant \delta_p(r) \) always holds.  

Note that the integers \( l_i \) are ordered as \( l_1 < l_2 < \cdots < l_{f_p} \). Moreover, whenever an element in the summand belongs to a block \( A_{i,j} \) with \( i > j \), its valuation does not contribute an \( e_p \). Consequently, the minimal number of \( e_p \) terms occurs when we choose the maximum possible number of elements strictly below the diagonal blocks. In this case, exactly \( \max k_{p,i} \) elements remain above the diagonal. Clearly, we have \( \max k_{p,i} = \delta_p(r) \), which completes the proof for \( i = r \).  

For the general case, we show that any diagonal \( i \)-minor satisfies the inequality. This follows immediately from the observation that if \( \tilde{\delta}_p \) is the function associated with the chosen principal \( i \)-minor, then  
\[
\tilde{\delta}_p(j) \geqslant \delta_p(j)
\]  
for any \( j \leqslant i \). This inequality follows from the Young diagram description of the function \( \delta \).  

 \end{proof}

For any type $\theta$, we associate the divisors $D_i$, for $i=1,\cdots, r$, as follows: 
$$D_i:= \sum_{p \in R} (i-e_p \delta_p(i))p.$$ 
\begin{defi} [The $\Gamma-$invariant Hitchin base]

We define a vector subspace $\al W^\theta$ in $\al W^\Gamma$  that depends on the type $\theta$ as follows: 
$$\al W^\theta:=\bigoplus_{i=1}^r H^0(X,K^i_X \otimes\mathcal{O}_X(D_i))^{\Gamma}.$$

\end{defi}

\begin{prop}\label{descends}  For any $i=1,\cdots, r$ we have $$ K^i_X \otimes\mathcal{O}_X(D_i)\simeq \pi^*\left( K^i_Y \otimes \mathcal{O}_Y \left( \sum_{q \in B} (i-\delta_q(i))q \right)\right).$$
\end{prop}
\begin{proof} We recall that for a covering $\pi:X \rightarrow Y$ and a reduced divisor $D=\{q\} \subset Y$ then one has:
$$  \pi^* (D)= \sum_{p\in \pi^{1}(q)} e_pp.
$$
 Using this fact one gets: 
 \begin{align*}
 K^i_X \otimes\mathcal{O}_X(D_i)&= K^i_X \otimes\mathcal{O}_X\left(\sum_{p \in R} (i-e_p \delta_p(i))p \right) \\
 &=\left( \pi^*(K_Y)\otimes\mathcal{O}_X(R) \right)^{i} \otimes \mathcal{O}_X\left(\sum_{p \in R} (i-e_p \delta_p(i))p \right)\\
 &= \pi^*(K_Y^{i})\otimes \mathcal{O}_X\left(\sum_{p \in R} i (e_p-1)+\sum_{p \in R} (i-e_p \delta_p(i))p  \right)\\
 &=\pi^*(K_Y^{i})\otimes \mathcal{O}_X\left(\sum_{p \in R} e_p(i-\delta_p(i))  \right) \\
 &=\pi^*\left( K^i_Y \otimes \mathcal{O}_Y \left( \sum_{q \in B} (i-\delta_q(i))q \right)\right).
  \end{align*}

\end{proof}
As a consequence of the above result we get for all $i$ an isomorphism:
$$\al W_i^\theta:=H^0(X, K^i_X \otimes\mathcal{O}_X(D_i))^{\Gamma} \simeq H^0 \left(Y, K^i_Y \otimes \mathcal{O}_X \left(\sum_{q \in B}\left(  i-\delta_q(i) \right)q \right) \right) .$$
To calculate the dimension of the $\Gamma-$invariant Hitchin base,  we recall the degree formula:
\begin{align*}
 \deg \left(  K^i_X \otimes\mathcal{O}_X(D_i)\right)&=n \deg \left(   K^{i}_Y \otimes\mathcal{O}_Y\left( \sum_{q\in B} (i-\delta_q(i))q \right)\right) \\
 &= 2n (g_Y-1)i+n \sum_{q\in B} (i-\delta_q(i)).
 \end{align*}
 
We set  $\mathcal{W}_i:=H^0(X,K^i_X \otimes\mathcal{O}_X(D_i))$ then we have:

\begin{align*} 
\dim \mathcal{W}_i^{\theta}=
\left\{ \begin{array}{ll}
1+2 (g_Y-1)i+\sum_{q\in B} (i-\delta_q(i)) & \text{if $i=1$} \\ \\
2 (g_Y-1)i+\sum_{q\in B} (i-\delta_q(i)) & \text{if $i \neq 1$ } \\
\end{array} \right.
\end{align*}
hence,
\begin{align*}
 \dim \mathcal{W}^{\theta}&= \sum_{i=1}^r \dim \mathcal{W}_i^{\Gamma, \theta} = r^2(g_Y-1)+1+ \sum_{q \in B} \sum_{i=1}^r (i-\delta_q(i)) \\
 &=\dim \al U_X^{\theta}(r,d).
\end{align*}

\section{Spectral curves and Hitchin fibers}
Let $q : P:=\bb P(K_X^{*}\oplus\mathcal{O}_X) \longrightarrow X$, the natural projection map, and $\mathcal{O}_X(1)$ the relative line bundle. As $q_*(K_X^{*}  \oplus\mathcal{O}_X)=K_X^{*}\oplus \mathcal{O}_X$ we take the natural section $x$ associated to $1$,  and we have $q_*(q^*(K_X)\otimes\mathcal{O}_X(1))= \mathcal{O}_X\oplus K_X$, denote by $y$ the section associated to $1$. Now for any element $s=(s_i)_i \in \al W^{\theta}$, we associate the polynomial 
$$ x^r+q^*(s_1)yx^{r-1}+...+q^*(s_{r-1})y^{r-1}x+q^*(s_{r})y^{r},$$
denote by $X_s$ the associated zero scheme and denote by $q:X_s \rightarrow X$ the restriction map. As the canonical bundle is equipped with a $\Gamma-$action, the scheme $P$ is equipped with $\Gamma-$action and this action induces a $\Gamma-$action on  $X_s$  since $s$ is $\Gamma-$invariant. Moreover, the map $q:X_s \rightarrow X$ is $\Gamma-$equivariant.  
\begin{lemm}\label{negative}
    For any $i\in\{1,\cdots,r\}$, we have $i-e_p\delta_p(i)\leqslant 0$.
\end{lemm}
\begin{proof}
    Let $f_p$ be the number of $k_{p,i}$ that are nonzero. Then clearly $e_p\geqslant f_p$. Moreover, from the definition of $\delta_p$, one sees that $i-f_p\delta_p\leqslant 0$. This implies $$i-e_p\delta_p(i)\leqslant i-f_p\delta_p(i)\leqslant 0.$$
\end{proof}
\begin{theo} For generic elements $s \in \al W^{\theta}$ the scheme $X_s$ is a integral curve of genus $r^2(g_X-1)+1$. Moreover, the ramification divisor of $q:X_s\ra X$ is contained in the pullback of $R$.
\end{theo}
\begin{proof} The proof is similar to \cite{BNR}. As the integrality is an open condition, it is sufficient to find an integral spectral curve. Consider the spectral curve $X_s$ given by the equation $x^r+s_r=0$, for some invariant section $s_r \in H^0(X,K^r(D_r))^{\Gamma}$. Now, if $s_r$ is not an $m-$th power for some section in $\mathcal{W}^{\theta}$ for $m>1$, then the curve $X_s$ is integral. This is generically true. \\
Let us calculate the genus of $X_s$. The curve sits inside the total space of the canonical line bundle $K_X$, hence 
\begin{align*}
    1-g_{X_s}=\chi(X_s,\mathcal{O}_{X_s})&=\chi(X,q_*(\mathcal{O}_{X_s}))\\
    &= \sum \chi(X,K^{i}_X)=- \deg(K_X) \frac{r(r-1)}{2}+r(1-g_X)\\
    &= -r^2 (g_X-1), 
\end{align*}
  and hence we get $g_{X_s}=r^2 (g_X-1)+1$.  
 Note that by the week Bertini theorem we can choose a section $s_r \in H^0(X,K^r(D_r))^{\Gamma}$ such that it has simple roots outside the ramification divisor $R$. Since the smoothness of the spectral curve outside $R$ is an open condition, we can take $X_s$ to be the curve of equation $x^r+s_r=0$. In particular, for general spectral curves $X_s$, we have  $Sing(X_s) \subset Ram(X_s/X) \subset q^{-1}(R)$. 
\end{proof}

We now determine the types where a general spectral curve is smooth. 
\begin{theo}\label{smooth}
    Let $\theta$ be a type. Then there exists an $s\in \al W^\theta$ such that the associated spectral curve $X_s$ is smooth if and only if for all $p\in R$ $$r-e_p\delta_p(r)\geqslant-1 \;\x{  or  }\; r-1-e_p\delta_p(r-1)=0.$$
\end{theo}
    \begin{proof}
    We use the Jacobian criterion to prove the smoothness. \\
Assume that $r-e_p\delta_p(r)= 0$, let $s_r\in H^0(X,K_X^r)$ such that $s_r(p)\neq 0$. The spectral curve $X_s$ associated to $(0,\cdots,0,s_r)$ has equation $$x^r+s_r(t)=0.$$ We see that its derivative  with respect to $x$ at any solution over $p$ is non zero. Hence $X_s$ is smooth.\\
Suppose now that $r-e_p\delta_p(r)= -1$ and let $s_r\in H^0(X,K_X^r)$ such that $s'_r(p)\neq 0$.
Note that such a global section exists since $H^0(X,K_X^r(-p))^\Gamma\sm H^0(X,K_X^r(-2p))^\Gamma$ is non-empty (use the fact that $K_X^r(-p)$ descends to $Y$; see Proposition \ref{descends}).  The spectral curve $X_s$ associated to $(0,\cdots,0,s_r)$ has equation $$x^r+s_r(t)=0.$$ We see that its derivative with respect to $t$ over $p$, which equals $s_r'(p)$, is not zero. Hence $X_s$ is smooth.\\

Assume now that $r-1-e_p\delta_p(r-1)=0$. Then we can choose a section $$s_{r-1}\in H^0(X,K_X^{r-1})$$ such that $s_{r-1}(p)\not=0$, and all the other sections $s_i=0$. Then the equation of the associated spectral curve over $p$ is $$x^r+s_{r-1}(p)x=0.$$ For the point $x=0$, the Jacobian criterion gives $$\dfrac{\partial}{\partial x}(x^r+s_{r-1}(p)x)=rx^{r-1}+s_{r-1}(p)=s_{r-1}(p)\not=0.$$ Hence this point is smooth. If $x$ is an other nonzero points, we have $x^{r-1}+s_{r-1}(p)=0$, but since $r\geqslant 2$, we deduce that $$rx^{r-1}+s_{r-1}(p)\not=0.$$ So $x$ is also smooth. \\ 
        Conversely, assume that both conditions are not satisfied. This means that for all spectral data $s$, we have $s_{r-1}(p)=s_r(p)=s_r'(p)=0$. It follows that the point $x=0, t=0$ is a singular point by the Jacobian criterion.
    \end{proof}
    \smallskip
\begin{rema}\label{Yangsmooth}
    A type $\theta$ verifies the first condition $r-e_p\delta_p(r)=0$ if and only if  its Yang diagram is rectangular with $e_p$ raws, in other wards it is of the form $$
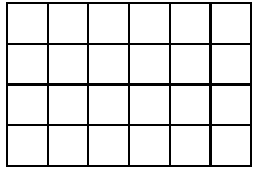
\begin{ytableau}
   \,& \, & \, & \, & \, & \,  \\
   \,& \, & \, & \, & \, & \,  \\
      \,& \, & \, & \, & \, & \,  \\
   \,& \, & \, & \, & \, & \,  \\
\end{ytableau} 
$$
Indeed, under this condition, we have the inequality $$r\leqslant f_p\delta_p(r) \leqslant e_p\delta
_p(r)=r,$$ hence $f_p=e_p$. This implies that the Yang diagram has $e_p$ rows and, since it has $r$ cells, the result follows. The converse is clear.\\
Now, if $\theta$ verifies the condition $r-e_p\delta_p(r)=-1$, then  its Yang diagram is of the form 
$$ 
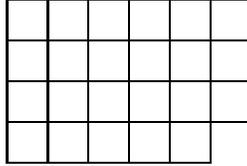
\begin{ytableau}
   \,& \, & \, & \, & \, & \,  \\
   \,& \, & \, & \, & \, & \,  \\
      \,& \, & \, & \, & \, & \,  \\
   \,& \, & \, & \, & \,  \\
\end{ytableau} 
$$
That is rectangular with just one missing cell. Indeed, we have $$r\leqslant f_p\delta_p(r)\leqslant e_p\delta_p(r)=r+1,$$
hence either we have $f_p\delta_p(r)=r$ or $r+1$. In the first case, we see that $$(e_p-f_p)\delta_p(r)=1,$$
which implies that the Yang diagram has only one column of length $e_p-1=f_p=r$. In the second case, we see that $f_p\delta_p(r)=e_p\delta_p(r)=r+1$, in particular $f_p=e_p$, and hence the Yang diagram has the desired form.  \\

For the second condition, namely $r-1=e_p\delta_p(r-1)$, the Yang diagram is given in the form 
$$
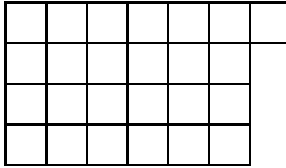
\begin{ytableau}
   \,& \, & \, & \, & \, & \, &\,  \\
   \,& \, & \, & \, & \, & \,  \\
      \,& \, & \, & \, & \, & \,  \\
   \,& \, & \, & \, & \, & \,  \\
\end{ytableau} $$
Indeed, Assuming $\delta_p(r)=\delta_p(r-1)$, we see that $$r\leqslant f_p\delta_p(r)\leqslant e_p\delta_p(r)=e_p\delta_p(r-1)=r-1,$$ which absurd. Hence $\delta_p(r)=\delta_p(r-1)+1$. Moreover, we have $$r-1= f_p\delta_p(r-1)\leqslant e_p\delta_p(r-1)=r-1,$$ hence $f_p=e_p$ and $r-e_p\delta_p(r)=1-e_p<0$. \\
\end{rema}
    \begin{defi}
        The types that verifies the conditions of Theorem \ref{smooth} are called \emph{smooth types}. The other types are called non-smooth. 
    \end{defi}
In the following, we deal only with the smooth types. We aim to study the singular case in an upcoming project.

\bigskip
\subsection*{Fibers of the Hitchin morphism} \noindent \smallskip

In \cite{BNR}, it is shown that the map $$T^*\al U_X(r,d)\lra \al U_X(r,d)\times \al W$$ given by the first projection and the Hitchin morphism, is dominant, and this is due to the existence of very stable vector bundles. The fibers of this map are identified, at least in the smooth case, to an open subset of the Picard group $\x{Pic}^c(X)$ of degree $c$ line bundles on the spectral curve, where $c$ is given by $-\deg (q_*\al O_{X_S})+d$(see \cite[Proposition 3.6]{BNR} for more details).\\

In the case of $\sigma-$invariant vector bundles of type $\tau$, the fibers are identified with an open subset of  the subvariety $\x{Pic}^{c,\tilde\tau}(X_s)$ of $\tilde \sigma-$invariant line bundles over the spectral curve, where the type $\tilde\tau$ is determined by the type $\tau$ (see \cite{Z}).
\smallskip

We fix $c=-\deg(q_*\al O_{X_s})+d$. 
In the following results, we need to use the notion of very stable bundles (for more details about this see \cite{Z4}). In a natural way, we generelize this notion to the $\Gamma-$invariant vector bundles.
\smallskip

A $\Gamma-$invariant  bundle $(E,\psi)$ is called \emph{very stable} if for any $\phi\in H^1(X,{\rm End}(E))^\Gamma$, we have $$\al H_\theta(\phi)=0\Rightarrow \phi=0.$$

\begin{prop}\label{nilpotent}
    The nilpotent cone $\Lambda^\theta\subset T^*\al U_X^{\Gamma,\theta}(r,d)$ is isotropic in $T^*\al U_X^{\Gamma,\theta}(r,d)$. In particular, the locus of very stable $\Gamma-$invariant vector bundles is dense in $\al U_X^{\Gamma,\theta}(r,d)$.
\end{prop}
\begin{proof}
    The proof is similar to \cite[Theorem $4.8$]{Z}.
\end{proof}

 \begin{theo}\label{main1}
    For any smooth type $\theta$ there exists a type $\tilde\theta$ of $\Gamma-$invariant line bundles on the spectral curve $X_s$, such that the  fiber $\al H_\theta^{-1}(s)$ is identified with a non-empty open subset of ${\rm Pic}^{c,\tilde \theta}(X_s)$ and $$\dim{\rm Pic}^{c,\tilde \theta}(X_s)=\dim\al U_X^{\Gamma,\theta}(r,d).$$ 
\end{theo}
\begin{proof}
Let $s\in \al W^\theta$ such that the associated spectral curve $X_s$ is smooth. Let $(E,\psi,\phi)\in\al H_\theta^{-1}(s)$. By \cite[Proposition $3.6$]{BNR}, the pair $(E,\phi)$ corresponds  to a unique line bundle $L$ over $X_s$ such that $q_*L=E$. Since the map $q:X_s\ra X$ is  $\Gamma-$equivariant map and $E$ is $\Gamma-$invariant, we deduce that $L$ is also $\Gamma-$invariant.

Now, we determine the types $\tilde\theta$ associated with $\theta$. The Yang diagram of $\theta$ has three forms, as given in Remark (\ref{Yangsmooth}).

In the first case, for any $p\in R$, the matrix of $\psi_{\gamma_p}$ is a diagonal matrix where each $e_p$-th root of unity appears $r/e_p$ times. Since zero is not a point in $X_s$ in this case, the action of $\Gamma$ on $X_s$ has no fixed points. Hence, there is only one trivial type which we denoted $\tilde\theta$.

Let $L$ be a $\Gamma$-invariant line bundle on $X_s$, and denote $E = q_*L$. Then,
\[
E_p = \bigoplus_{x\in q^{-1}(p)} L_x.
\]
Since none of the points $x \in q^{-1}(p)$ are fixed, the linearization on $E_p$ takes the form:
\[
\begin{pmatrix}
     A & 0 & \cdots  & 0\\
     0 & A & 0 &  \vdots\\
     \vdots & 0 & \ddots &  0 \\
     0 &   \cdots & 0 &  A
\end{pmatrix},
\]
where each block $A$ is an $e_p \times e_p$  matrix with ones under the diagonal and at the top right corner. Specifically, 
\[
A = \begin{pmatrix}
   0 & \cdots & 0 & 1 \\
   1  & 0  & \ddots & 0   \\
   \vdots & \ddots & 0 & \vdots\\
   0 & \cdots & 1 & 0
\end{pmatrix}
\sim
\begin{pmatrix}
     \xi_{p}^1 & 0 & \cdots & 0 \\
     0 & \xi_{p}^2  & 0 & \vdots \\
     \vdots & 0 & \ddots & 0 \\
     0 & \cdots & 0 & \xi_{p}^{e_p}
\end{pmatrix}.
\]
This shows that each eigenvalue of $\psi_{\gamma_p}$ appears exactly $r/e_p$ times.

In the second case,  namely when $r-\delta_p(r)e_p=-1$, the type $\theta$ satisfies
\[
\theta_p = \left(\frac{r+1}{e_p}, \dots, \frac{r+1}{e_p} -1, \dots, \frac{r+1}{e_p}\right).
\]
In other words, all eigenvalues except one have multiplicity $\frac{r+1}{e_p}$, while the remaining eigenvalue, say for simplicity $\xi_{p}^{e_p-1}$, has multiplicity $\frac{r+1}{e_p} -1$.\\
Now, the fiber of $q:X_s\ra X$ over $p$ has one fixed point $\tilde{p}$ with multiplicity $e_p-1$, and the remaining $r-e_p+1=(\delta_p(r)-1)e_p$ points are not fixed by $\Gamma_p$, this produces $\delta_p(r)-1$ orbits, each one produces all the eigenvalues from $\xi_p$ to $\xi_p^{e_p}$ as explained with the matrix $A$ in the last case.\\ 
Let $L$ be a line bundle with trivial type, that's $\tilde\theta_{\tilde p}=(1,0,\cdots,0)$. The fiber of the bundle $E=q_*L$ over $p$ can be identified with $$p=\left(\bigoplus_{x\in q^{-1}(p)\sm \{\tilde p\}} L_x\right)\oplus L_{(e_p-1)\tilde{p}},$$
where $L_{(e_p-1)\tilde{p}}=L\otimes \al O_{\tilde p}/\ak m^{e_p}$ is free rank one  module over $\al O_{\tilde p}/\ak m^{e_p}\cong\bb K[t]/\gen{t^{e_p}}$, where $\ak m$ is the maximal ideal associated to $\tilde p$.  One sees that the action of $\Gamma_p$ on this module is given by $$(a_0,\cdots,a_{e_p-1})\mapsto (a_0,\xi_p a_1,\cdots, \xi_p^{e_p-1}a_{e_p-1}).$$
It follows that this produces the eigenvalues $\xi_p,\cdots,\xi^{e_p-1}$. Finally, the type of $E$ is exactly what we looking for.\\

In the third case, the type $\theta$ satisfies
\[
\theta_p = \left(\frac{r-1}{e_p}, \dots, \frac{r-1}{e_p} +1, \dots, \frac{r-1}{e_p}\right).
\]
In other words, all eigenvalues except one have multiplicity $\frac{r-1}{e_p}$, while the remaining eigenvalue, say $\xi_p^j$, has multiplicity $\frac{r-1}{e_p} +1$.

Note that $r - e_p\delta_p(r) = 1 - e_p < 0$ (see Remark (\ref{Yangsmooth})). Hence, the point zero belongs to $X_s$, precisely to the fiber over $p$, and it is the only fixed point under the $\Gamma_p$ action.

Now, let $\tilde \theta$ be a type of $\Gamma$-invariant line bundle on $X_s$ defined by
$$
\tilde\theta_p = (0, \dots, 0, 1, 0, \dots, 0),
$$
where $1$ is in the $j$-th position. Let $L$ be a $\Gamma$-invariant line bundle on $X_s$ of type $\tilde\theta$ and define $E = q_*L$. As observed earlier, 
\[
E_p = \bigoplus_{x\in q^{-1}(p)} L_x,
\]
and all points in the fiber except one are not fixed. Hence, the matrix of the linearization on $E_p$ is given by:
$$
\begin{pmatrix}
     \xi_{p,j} & 0 & \cdots & \cdots & 0\\
     0 & A & 0 & \cdots & \vdots \\
     0 & 0 & \ddots & 0 & 0 \\
     \vdots & \cdots & 0 & \ddots & 0\\
     0 & 0 & \cdots & 0 & A
\end{pmatrix},
$$
where $A$ is as defined earlier. We conclude that all eigenvalues have multiplicity $\frac{r-1}{e_p}$, except for $\xi_{p,j}$, which has multiplicity $\frac{r-1}{e_p} + 1$.

Now, we have ${\rm Pic}^{c,\tilde\theta}(X_s)\cong{\rm Pic}^{c/n}(X_s/\Gamma)$ whose dimension equals the genus of $X_s/\Gamma$. \\ In the first case, the cover $X_s\ra X_s/\Gamma$ has no ramification points, by Hurwitz formula, the smooth curve $X_s/\Gamma$ has genus equal to 
\begin{align*}
    g_{X_s/\Gamma}&=\frac{1}{n}(g_{X_s}-1)=\frac{r^2}{n}(g_X-1) \\ &=r^2(g_Y-1)+\frac{r^2m}{2n} \\& =r^2(g_Y-1)+\frac{r^2}{2n}\sum_{p\in R}e_p-1 &(\x{Since }m=\deg R) \\& =r^2(g_Y-1)+1+\frac{r^2}{2}\sum_{q\in B}\frac{(e_q-1)}{e_q} \\ &= \dim \al U_X^{\Gamma,\theta}(r,d) &(\x{By Theorem (\ref{dim}}))
\end{align*}

In the second case, over each point $p\in R$, there is one fixed point of ramification index equals to $e_p-1$. It follows that  the genus of  the  curve $X_s/\Gamma$  is 
 \begin{align*}
    g_{X_s/\Gamma}&=\frac{1}{n}(g_{X_s}-1)+1-\frac{1}{2n}\sum_{p\in R}e_p-1\\ &=r^2(g_Y-1)+1+\frac{r^2m}{2n}-\frac{1}{2n}\sum_{p\in R}e_p-1 \\ &=r^2(g_Y-1)+1+\frac{1}{2n}\sum_{p\in R}r^2(e_p-1)-(e_p-1) \\& =r^2(g_Y-1)+1+\frac{1}{2}\sum_{q\in B}\frac{(r^2-1)(e_p-1)}{e_p}  \\ &= \dim \al U_X^{\Gamma,\theta}(r,d)
\end{align*}

 In the third case, there is only one  fixed point in $X_s$ over each $p\in R$ of ramification index $e_p-1$, hence by Hurwitz formula we get $$2(g_{X_s}-1)=2n(g_{X_s/\Gamma}-1)+\sum_{p\in R}e_p-1.$$ It follows that
 \begin{align*}
    g_{X_s/\Gamma}&=\frac{1}{n}(g_{X_s}-1)+1-\frac{1}{2n}\sum_{p\in R}e_p-1\\ &=r^2(g_Y-1)+1+\frac{r^2m}{2n}-\frac{1}{2n}\sum_{p\in R}e_p-1 \\ &=r^2(g_Y-1)+1+\frac{1}{2n}\sum_{p\in R}r^2(e_p-1)-(e_p-1) \\& =r^2(g_Y-1)+1+\frac{1}{2}\sum_{q\in B}\frac{(r^2-1)(e_p-1)}{e_p}  \\ &= \dim \al U_X^{\Gamma,\theta}(r,d)
\end{align*}

\end{proof}
\begin{rema}
    Note that the smooth curve $X_s/\Gamma$ can be seen as the normalization of the singular spectral curve $Y_s$
 over $Y$. \end{rema}
\begin{theo}\label{main2}
    The pushforward rational map $$q_*:{\rm Pic}^{c,\tilde \theta}(X_s)\dashrightarrow \al U_X^{\Gamma,\theta}(r,d)$$ is dominant. In particular, the moduli space $\al U_X^{\Gamma,\theta}(r,d)$ is connected.
\end{theo}
\begin{proof}
    Consider the map $$\sr H_\theta:T^*\al U_X^{\Gamma,\theta}(r,d)\lra \al U_X^{\Gamma,\theta}(r,d)\times \al W^\theta.$$
This map is dominant. Indeed by Proposition (\ref{nilpotent}), there exists a very stable $\Gamma-$invariant bundle of any given type, then the result follows from the dimension theorem. Now, if ${\rm pr}_1:\al U_X^{\Gamma,\theta}(r,d)\times \al W^\theta\ra \al U_X^{\Gamma,\theta}(r,d)$ and ${\rm pr}_2:\al U_X^{\Gamma,\theta}(r,d)\times \al W^\theta\ra W^\theta$ are  the first and second  projections, one sees that $$\sr H_\theta\circ{\rm pr}_1:\al H^{-1}(s)\lra\al U_X^{\Gamma,\theta}(r,d) $$ is dominant too. On the other hand, we have 
$$(\sr H_\theta\circ{\rm pr}_2)^{-1}(s)=\al H_\theta^{-1}(s).$$ 
But the fiber $\al H_\theta^{-1}(s)$ is identified with a non-empty open subset of ${\rm Pic}^{c,\tilde \theta}(X_s)$ by Theorem (\ref{main1}). The result follows.
\end{proof}

\section{Appendix: Seshadri correspondence} \label{Seshadri correspondence}

Seshadri in \cite{Seshadri} gave a correspondence between the moduli space of invariant bundles over a curve equipped with a group action and the moduli space of parabolic bundles over the quotient curve. In this section, we recall the definition of parabolic bundles and the Seshadri correspondence. The maine references of this appendix are Bhosle  \cite{bhosle} and Mheta-Seshadri \cite{Seshadri}.\\

Let $Y$ be a smooth projective algebraic curve. 
\begin{defi}[Parabolic Bundles] Let $E$ be a vector bundles on $Y$, a quasiparabolic structure on $E$ supported at a point $y\in Y$ is a decreasing sequence of linear subspaces called the flag 

$$ E_y = F_y^1E  \supseteq  F_y^2 E \supseteq \ldots \supseteq F^{\ell_y}_y E \supseteq F_y^{\ell_y+1} E =\{0\},$$
where $\ell_y$ is the length of the flag, we define the flag type by the sequence $$k_{y,i} := \dim F_y^{i}E -  \dim F_y^{i+1}E.$$  
A parabolic structure in $E$ is given by a quasiparabolic structure together with a sequence of real numbers, called the parabolic weights:  $$0 \leq \alpha_{y,1} \leq \alpha_{y,2} \leq \dots \leq \alpha_{y,\ell_y} < 1.$$
We denote a bundle $E$ equipped with a parabolic structure by $E_*$.
\end{defi} 

Let $B=\{y_1,y_2, \dots ,  y_N \}$ a finit subset in $Y$. A parabolic type on $B$ is the data of flag type and  weights at each point in $B$.
\begin{defi} Let $E$ a vector bundle on $Y$ with a fixed parabolic structure on $B$ of a fixed parabolic type. Then we define the parabolic degree
$$\deg_{par}(E):=\deg(E)+\sum_{y \in B} \sum_{i=1}^{\ell_y} \alpha_{y,j} \ k_{y,j},$$
and the parabolic slop
$$
\mu_{par}(E):= \frac{\deg_{par}(E)}{rk(E)}.
$$  
\end{defi}
Let $E_*$ be a parabolic bundle over $B$ of a fixed type. $E_*$ is said to be semistable (res. stable) if for any subbundle $F$ of $E$ equipped with the natural induced structure satisfy the slop inequality 
\[
    \mu_{par}(F) \leqslant \mu_{par}(E) \quad (\text{resp. } \mu_{par}(F) < \mu_{par}(E)).
\]
In \cite{MS} Mehta and Seshadri construct the moduli space $\mathcal{M}^{par}_{Y,B}(r,d)$ of semistable parabolic bundles over $Y$ with fixed rank, degree and parabolic type, and show that is a projective variety. They also proved that the  locus of stable parabolic  bundles is an open smooth subset.
\subsection*{Seshadri correspondence} \noindent \smallskip

Let $\pi: X \to Y $ be a Galois covering of smooth curves with Galois group $\Gamma$, and let $E$ be a $ \Gamma-$invariant vector bundle over $X$. For a point $p \in R$ and $\xi \in \Gamma_p \simeq \mathbb{Z}_{e_p}
$, a generator of the isotropy group, we get an eigenspace decomposition of $E_p = \bigoplus_{i=1}^{f_p} E_{\xi^{n_i}}$. After a finite sequence of Hecke modifications with respect to the eigensubspaces, one obtains a subsheaf $E'$ is a subsheaf of $E$ which is locally free and of the same rank, where the $\Gamma_p$-action is trivial. Now, if we apply this to any point of the ramification divisor, we get a vector bundle with trivial action, hence it descends to a vector bundle over $Y$ by Kumpf's Lemma. In fact, the vector bundle $E'$ is given by $\pi^* \left( \pi_*^{\Gamma}( E )\right)$. The vector bundle $\pi^* \left( \pi_*^{\Gamma}( E )\right)$ is equipped with the induced direct sum decomposition, hence the invariant pushforward is equipped with a this decomposition.  \smallskip

The vector bundle $\pi_*^{\Gamma}( E )$ over $Y$ is equipped with a parabolic structure over the branched locus $B=\pi(R)$, as follows: 
\begin{itemize}
    \item  The quasiparabolic structure:
$$ F_y^{j}\left( \pi_*^{\Gamma}( E ) \right):= \bigoplus_{i=0}^{j} E'_{\xi^{n_j}}$$ 
 \item The parabolic weights is given by:  $$\alpha_{y,j}:= \frac{j}{e_p}.
$$
\end{itemize}
This associates to a type $\theta$ a parabolic type and denotes by $\mathcal{M}^{par}_{Y,B}(r,d)$ the space of parabolic bundles of this type. 
\begin{theo}[Seshadri] The invariant pushforward map 
$$
\begin{array}{ccccc}  
 &\pi_*^{\Gamma} : &\al U_X^{\Gamma,\theta}(r,d)   & \longrightarrow &  \mathcal{M}^{par}_{Y,B}(r,d)  \\ 
& & (E, \psi)  & \longmapsto & \pi_*^{\Gamma}(E)_*
\end{array}
$$
is an isomorphism of varieties.
\end{theo}


\bibliographystyle{alpha}
\bibliography{bib}
\end{document}